\documentclass[nointlimits,11pt,oneside]{amsart}
\usepackage{amssymb,cases,enumitem, verbatim}
\usepackage{xcolor}
\usepackage{hyperref}
\usepackage[T1]{fontenc}
\usepackage[utf8]{inputenc}

\hypersetup{
	colorlinks=true,
	linkcolor=blue,
	citecolor=blue,
	filecolor=green,
	urlcolor=cyan,
	bookmarks=true
}

\makeatletter
\renewcommand*{\eqref}[1]{%
	\hyperref[{#1}]{\textup{\tagform@{\ref*{#1}}}}%
}
\makeatother 

\setlist[enumerate,1]{label={\textup{(\roman*)}}}

\usepackage[%
	a4paper,
	total={16cm,23cm},
	left=3cm, top=3cm,
	marginparsep=2pt]
{geometry}

\theoremstyle{plain}
\newtheorem{theorem}{Theorem}[section]

\newtheorem{corollary}[theorem]{Corollary}
\newtheorem{proposition}[theorem]{Proposition}
\theoremstyle{definition}

\newtheorem{definition}[theorem]{Definition}

\numberwithin{equation}{section}

\DeclareMathOperator\supp{supp}

\def\L1loc{L^1_{\text{loc}}}

\begin{document}

\title{Absolute continuity of the (quasi)norm in rearrangement-invariant spaces}
\author{Dalimil Pe{\v s}a}

\address{Dalimil Pe{\v s}a, Technische Universität Chemnitz, Faculty of Mathematics, 09107 Chemnitz, Germany
	and
	Department of Mathematical Analysis, Faculty of Mathematics and Physics, Charles University, Sokolovsk\'a~83, 186~75 Praha~8, Czech Republic}
\email{dalimil.pesa@mathematik.tu-chemnitz.de}
\urladdr{0000-0001-6638-0913}

\subjclass[2010]{46E30, 46A16}
\keywords{Banach function space, quasi-Banach function space, rearrangement-invariant space, absolute continuity of the norm, representation theorem}

\thanks{This research was supported by the grant 23-04720S of the Czech Science Foundation.}

\begin{abstract}
This paper explores the interactions of absolute continuity of the (quasi)norm with the concepts that are fundamental in the theory of rearrangement-invariant (quasi-)Banach function spaces, such as the Luxemburg representation or the Hardy--Littlewood--P{\' o}lya relation. In order to prove our main results, we give an explicit construction of a particularly suitable representation quasinorm (which is not necessarily unique) and develop several new tools that we believe to be of independent interest. As an application of our results, we characterise the subspace of functions having absolutely continuous quasinorms in weak Marcinkiewicz spaces.
\end{abstract}

\date{\today}

\maketitle

\makeatletter
   \providecommand\@dotsep{2}
\makeatother

\section{Introduction}
Absolute continuity of the (quasi)norm is one of the fundamental concepts in the theory of (quasi-)Banach function spaces (we provide the relevant definitions in Sections~\ref{SectionFunctionNormsQuasinorms} and \ref{SecACqN}). The importance of this concept comes from two directions. 

Firstly, it is deeply connected to the functional-analytic properties of said spaces. Indeed, it is well known that a (quasi-)Banach function space $X$ (over an appropriate underlying measure space) is separable if and only if it has absolutely continuous (quasi)norm, while a Banach function space $X$ is reflexive if and only if both itself and its associate space $X'$ have this property. This is a consequence of the fact that absolute continuity of the (quasi)norm is the property that determines whether a pointwise convergence translates to convergence in the given space (in the sense of the abstract version of the Lebesgue dominated convergence theorem that we present as Proposition~\ref{PropDomConv}), as this is precisely the crux of the matter, both for separability and reflexivity. For details, see e.g.~\cite[Chapter~1, Sections~3, 4, and 5]{BennettSharpley88} and \cite[Section~3.3]{NekvindaPesa24}. 

Secondly, it has turned out to be relevant in many areas of applications. For example, it has long been known that absolute continuity of the norm is important in understanding various forms of compactness in Banach function spaces. In \cite{LuxemburgZaanen62}, it has been presented as a key component of various results about weak-type compactness. For separable Banach function spaces, it has been shown in \cite{Zabreiko66} that a set of functions is precompact if and only if it is locally precompact in measure and the functions have uniformly absolutely continuous norms (see also \cite{CaetanoGogatishvili16} for extension of this result to the class of quasi-Banach function spaces). Since the unit balls of Sobolev spaces based on Banach function spaces (over bounded Lipschitz domains) are always compact in measure (as follows easily from the Rellich--Kondrachov theorem), this characterisation has been applied with much success to studying the compactness of Sobolev embeddings in \cite{CurberaRicker07}, \cite{KermanPick08}, and \cite{Pustylnik06}. Similar principles also apply to other operators, which means that absolute continuity of the norm has been employed in characterising compactness of kernel-type integral operators in \cite{LuxemburgZaanen62} and \cite{RafeiroSamko08} and also the Hardy operator in \cite{LaiPick93} and its modifications including weights and suprema in \cite{EdmundsGurka94} and \cite{PerneckaPick13}, respectively. Another related concept are the almost compact embeddings (sometimes also called absolutely continuous embeddings) that appeared naturally in \cite{KermanPick08} as a tool that can provide compactness and have later been studied in \cite{Fernandez-MartinezManzano10} and \cite{Slavikova12}. Finally, as an example not related to compactness, we would like to include a recent characterisation of functions vanishing at the boundary in some Sobolev spaces that has been obtained in \cite{NekvindaTurcinova24}.

This paper aims to explore the interactions of absolute continuity of the (quasi)norm with the concepts that are fundamental in the theory of rearrangement-invariant (quasi-)Banach function spaces, such as the Luxemburg representation or the Hardy--Littlewood--P{\' o}lya relation. Indeed, when considering the absolute continuity of the (quasi)norm in a rearrangement-invariant space several questions naturally appear, answers to which cannot be found in the literature (to our great surprise). The two main questions we answer in this paper are:

\begin{enumerate}[label={\textup{\textbf{(\Alph*)}}}]
	\item \label{QuestionRepre} \textbf{Representation of the absolute continuity of the (quasi)norm.} Does it hold that a function $f$ has absolutely continuous (quasi)norm in a given rearrangement-invariant (quasi-)Banach function space if and only if $f^*$, the non-increasing rearrangement of $f$ (see Definition~\ref{DefNIR} for details), has the same property in the corresponding representation space?
	
	The answer to this question is positive, at least for the representation space we construct in Section~\ref{SectionRepreCorrect}; this is the content of Theorem~\ref{TheoremRepreACqN}.
	
	\item \label{QuestionOrder} \textbf{Preservation by weaker orderings.}	
	Let $X$ be a (quasi-)Banach function space and denote by $X_a$ the subspace of functions having absolutely continuous (quasi)norm. Since $X$ has the lattice property, i.e.~$\lvert f \rvert \leq \lvert g \rvert$ $\mu$-a.e.~implies $\lVert f \rVert_X \leq \lVert g \rVert_X$, it is clear that when $\lvert f \rvert \leq \lvert g \rvert$ and $g \in X_a$, then also $f \in X_a$. However, in rearrangement-invariant spaces, the (quasi)norm preserves also some weaker orderings of functions. The question thus naturally appears whether the above mentioned principle translates to those weaker orderings. Specifically:
	\begin{enumerate}
		\item Every rearrangement-invariant (quasi-)Banach function space $X$ has the property that $f^* \leq g^*$ implies $\lVert f \rVert_X \leq \lVert g \rVert_X$, see e.g.~\cite[Corollary~3.5]{MusilovaNekvinda24}. Does this mean that having $f^* \leq g^*$ and $g \in X_a$ implies $f \in X_a$?
		\item Some rearrangement-invariant quasi-Banach function spaces, including (but not limited to) all rearrangement-invariant Banach function spaces, have the property that $f \prec g$ implies $\lVert f \rVert_X \leq \lVert g \rVert_X$ (where $\prec$ denotes the Hardy--Littlewood--P{\' o}lya relation, see Definition~\ref{DefHLPR} and the following discussion for details and context). Provided that $X$ is such a space, does it mean that having $f \prec g$ and $g \in X_a$ implies $f \in X_a$?
	\end{enumerate}
	
	Answers to both of the sub-questions are again affirmative; this is the content of Corollary~\ref{CorollaryACRI} and Theorem~\ref{TheoremACHLP}, respectively.
\end{enumerate}

It is worth noting that while the answers to our questions are rather natural, their proofs are far from straightforward, especially when it comes to the interaction of the absolute continuity of the (quasi)norm with the Hardy--Littlewood--P{\' o}lya relation. Furthermore, the natural positive answer to the question \ref{QuestionRepre} is in some situations true only for some of the possible choices of the representation functional, which we find quite surprising.

Building upon the recent advances in the theory of rearrangement-invariant quasi-Banach function spaces in \cite{MusilovaNekvinda24} and \cite{NekvindaPesa24}, we were able to prove our results in this more general context.
However, our results are not covered in the existing literature even for the classical case of rearrangement-invariant Banach function spaces. This forms a rather surprising gap in the classical theory that we now fill in.

Finally, in order to answer the question \ref{QuestionOrder} we had to develop some new tools that are based on our representation result (Theorem~\ref{TheoremRepreACqN}) and that we believe to be of independent interest. To justify this claim, we present an application where we use those tools to characterise the subspace of functions having absolutely continuous quasinorms in weak Marcinkiewicz spaces. In this connection let us recall that, for a given r.i.~quasi-Banach function space, this subset is often trivial (in the sense that it either coincides with the space itself, or, conversely, contains only the zero function); however, the pivotal instance of a class of spaces for which this subspace is non-trivial is exactly that of the weak Marcinkiewicz spaces.

\section{Preliminaries}

The purpose of this section is to establish the theoretical background that serves as the foundation for our research. The definitions and notation is intended to be as standard as possible. The usual reference for most of this theory is \cite{BennettSharpley88}.

Throughout this paper we will denote by $(\mathcal{R}, \mu)$, and occasionally by $(\mathcal{S}, \nu)$, some arbitrary (totally) $\sigma$-finite measure space. Given a $\mu$-measurable set $E \subseteq \mathcal{R}$ we will denote its characteristic function by $\chi_E$. By $\mathcal{M}(\mathcal{R}, \mu)$ we will denote the set of all extended complex-valued $\mu$-measurable functions defined on $\mathcal{R}$. As is customary, we will identify functions that coincide $\mu$-almost everywhere. We will further denote by $\mathcal{M}_0(\mathcal{R}, \mu)$ and $\mathcal{M}_+(\mathcal{R}, \mu)$ the subsets of $\mathcal{M}(\mathcal{R}, \mu)$ containing, respectively, the functions finite $\mu$-almost everywhere and the non-negative functions.

When there is no risk of confusing the reader, we will abbreviate $\mu$-almost everywhere, $\mathcal{M}(\mathcal{R}, \mu)$, $\mathcal{M}_0(\mathcal{R}, \mu)$, and $\mathcal{M}_+(\mathcal{R}, \mu)$ to $\mu$-a.e., $\mathcal{M}$, $\mathcal{M}_0$, and $\mathcal{M}_+$, respectively.

By a support of a given function $f \in \mathcal{M}(\mathcal{R}, \mu)$ we mean the set where it is non-zero, i.e.
\begin{equation*}
	\supp f =\left \{x \in \mathcal{R}; \; \lvert f \rvert> 0  \right \}.
\end{equation*}

When $X$ is a set and $f, g: X \to \mathbb{C}$ are two maps satisfying that there is some positive and finite constant $C$, depending only on $f$ and $g$, such that $\lvert f(x) \rvert \leq C \lvert g(x) \rvert$ for all $x \in X$, we will denote this by $f \lesssim g$. We will also write $f \approx g$, or sometimes say that $f$ and $g$ are equivalent, whenever both $f \lesssim g$ and $g \lesssim f$ are true at the same time. We choose this general definition because we will use the symbols ``$\lesssim$'' and ``$\approx$'' with both functions and functionals. 

When $X, Y$ are two topological linear spaces, we will denote by $Y \hookrightarrow X$ that $Y \subseteq X$ and that the identity mapping $I : Y \rightarrow X$ is continuous.

As for some special cases, we will denote by $\lambda^n$ the classical $n$-dimensional Lebesgue measure, with the exception of the $1$-dimensional case where we simply write $\lambda$. We will further denote by $m$ the counting measure over $\mathbb{N}$. When $p \in (0, \infty]$ we will denote by $L^p$ the classical Lebesgue space (of functions in $\mathcal{M}(\mathcal{R}, \mu)$), defined for finite $p$ as the set
\begin{equation*}
	L^p = \left \{ f \in \mathcal{M}(\mathcal{R}, \mu); \; \int_{\mathcal{R}} \lvert f \rvert^p \: d\mu < \infty \right \},
\end{equation*}
equipped with the customary (quasi-)norm
\begin{equation*}
	\lVert f \rVert_p = \left ( \int_ {\mathcal{R}} \lvert f \rvert^p \: d\mu \right )^{\frac{1}{p} },
\end{equation*}
and through the usual modifications for $p=\infty$. In the special case when $(\mathcal{R}, \mu) = (\mathbb{N}, m)$ we will denote this space by $l^p$. Note that in this paper we consider $0$ to be an element of $\mathbb{N}$.

\subsection{Non-increasing rearrangement} \label{SectionNon-increasingRearrangement}
We now present the crucial concept of the non-increasing rearrangement of a function and state some of its properties that will be important for our work. We proceed in accordance with \cite[Chapter~2]{BennettSharpley88}.

\begin{definition} \label{DefNIR}
	The distribution function $f_*$ of a function $f \in \mathcal{M}(\mathcal{R}, \mu)$ is defined for $s \in [0, \infty)$ by
	\begin{equation*}
		f_*(s) = \mu(\{ t \in \mathcal{R}; \; \lvert f(t) \rvert > s \}).
	\end{equation*}	
	The non-increasing rearrangement $f^*$ of the said function is then defined for $t \in [0, \infty)$ by
	\begin{equation*}
		f^*(t) = \inf \{ s \in [0, \infty); \; f_*(s) \leq t \}.
	\end{equation*}
\end{definition}

For the basic properties of the distribution function and the non-increasing rearrangement, with proofs, see \cite[Chapter~2, Proposition~1.3]{BennettSharpley88} and \cite[Chapter~2, Proposition~1.7]{BennettSharpley88}, respectively. We consider those properties to be classical and well known and we will be using them without further explicit reference.

An important concept used in the paper is that of equimeasurability.

\begin{definition} \label{DEM}
	We say that the functions $f \in \mathcal{M}(\mathcal{R}, \mu)$ and $g \in \mathcal{M}(\mathcal{S}, \nu)$ are equimeasurable if $f^* = g^*$, where the non-increasing rearrangements are computed with respect to the appropriate measures.
\end{definition}

It is not hard to show that two functions are equimeasurable if and only if their distribution functions coincide too (with each distribution function being considered with respect to the appropriate measure).

A very significant classical result is the Hardy--Littlewood inequality. For proof, see for example \cite[Chapter~2, Theorem~2.2]{BennettSharpley88}.

\begin{theorem} \label{THLI}
	It holds for all $f, g \in \mathcal{M}$ that
	\begin{equation*}
		\int_\mathcal{R} \lvert fg \rvert \: d\mu \leq \int_0^{\infty} f^*g^* \: d\lambda.
	\end{equation*}
\end{theorem}

It follows directly from this result that it holds for every $f,g \in \mathcal{M}$ that
\begin{equation} \label{HLI_sup}
	\sup_{\substack{\tilde{g} \in \mathcal{M} \\ \tilde{g}^* = g^*}} \int_{\mathcal{R}} \lvert f \tilde{g} \rvert \: d\mu \leq \int_0^{\infty} f^*g^* \: d\lambda.
\end{equation}
This motivates the definition of resonant measure spaces as those spaces where we have equality in \eqref{HLI_sup}.

\begin{definition}
	A $\sigma$-finite measure space $(\mathcal{R}, \mu)$ is said to be resonant if it holds for all $f, g \in \mathcal{M}(\mathcal{R}, \mu)$ that
	\begin{equation*}
		\sup_{\substack{\tilde{g} \in \mathcal{M} \\ \tilde{g}^* = g^*}} \int_\mathcal{R} \lvert f \tilde{g} \rvert \: d\mu = \int_0^{\infty} f^* g^* \: d\lambda.
	\end{equation*}
\end{definition}

The property of being resonant is a crucial one. Luckily, there is a straightforward characterisation of resonant measure spaces which we present below. For proof and further details see \cite[Chapter~2, Theorem~2.7]{BennettSharpley88}.

\begin{theorem} \label{TheoremCharResonance}
	A $\sigma$-finite measure space is resonant if and only if it is either non-atomic or completely atomic with all atoms having equal measure.
\end{theorem}

Besides the non-increasing rearrangement, we will also need the concept of elementary maximal function, sometimes also called the maximal non-increasing rearrangement, which is defined as the Hardy transform of the non-increasing rearrangement.

\begin{definition}
	The elementary maximal function $f^{**}$ of $f \in \mathcal{M}$ is defined for $t \in (0, \infty)$ by 
	\begin{equation*}
		f^{**}(t) = \frac{1}{t} \int_0^{t} f^*(s) \: ds.
	\end{equation*} 
\end{definition}

Finally, we will need the following simple observation. For completeness, we also provide the short proof.

\begin{proposition} \label{PropACR}
	Let $f \in \mathcal{M}$ satisfy
	\begin{equation} \label{PropACR:E1}
		\lim_{t \to \infty} f^*(t) = 0
	\end{equation}
	and assume that the sequence $E_k \subseteq \mathcal{R}$ satisfies $\chi_{E_k} \to 0$ $\mu$-a.e.~as $k \to \infty$. Then
	\begin{align*}
		\lim_{k \to \infty} (f \chi_{E_k})^*(t) &= 0 & \text{for } t \in (0, \infty).
	\end{align*}
\end{proposition}

\begin{proof}
	Fix $f$ and $E_k$ as in the assumptions and arbitrary $\varepsilon > 0$. By \eqref{PropACR:E1}, we have for $F \subseteq \mathcal{R}$ given by
	\begin{equation*}
		F = \{ \lvert f \rvert > \varepsilon\}
	\end{equation*}
	that
	\begin{equation*}
		\mu(F) = \lambda(\{ f^* > \varepsilon\}) < \infty.
	\end{equation*}
	Hence, for any $t \in (0,\infty)$ there is an $k_0$ such that for every $k \geq k_0$ we have 
	\begin{equation*}
		\lambda(\{ (f \chi_{E_k})^* > \varepsilon \}) = \mu(F\cap E_k) < t,
	\end{equation*}
	which implies that $(f \chi_{E_k})^*(t) \leq \varepsilon$.
\end{proof}

\subsection{Banach function norms and quasinorms} \label{SectionFunctionNormsQuasinorms}

\begin{definition}
	Let $\lVert \cdot \rVert : \mathcal{M}(\mathcal{R}, \mu) \rightarrow [0, \infty]$ be a mapping satisfying $\lVert \, \lvert f \rvert \, \rVert = \lVert f \rVert$ for all $f \in \mathcal{M}$. We say that $\lVert \cdot \rVert$ is a Banach function norm if its restriction to $\mathcal{M}_+$ satisfies the following axioms:
	\begin{enumerate}[label=\textup{(P\arabic*)}, series=P]
		\item \label{P1} it is a norm, in the sense that it satisfies the following three conditions:
		\begin{enumerate}[ref=(\theenumii)]
			\item \label{P1a} it is positively homogeneous, i.e.\ $\forall a \in \mathbb{C} \; \forall f \in \mathcal{M}_+ : \lVert a f \rVert = \lvert a \rvert \lVert f \rVert$,
			\item \label{P1b} it satisfies $\lVert f \rVert = 0 \Leftrightarrow f = 0$  $\mu$-a.e.,
			\item \label{P1c} it is subadditive, i.e.\ $\forall f,g \in \mathcal{M}_+ \: : \: \lVert f+g \rVert \leq \lVert f \rVert + \lVert g \rVert$,
		\end{enumerate}
		\item \label{P2} it has the lattice property, i.e.\ if some $f, g \in \mathcal{M}_+$ satisfy $f \leq g$ $\mu$-a.e., then also $\lVert f \rVert \leq \lVert g \rVert$,
		\item \label{P3} it has the Fatou property, i.e.\ if  some $f_n, f \in \mathcal{M}_+$ satisfy $f_n \uparrow f$ $\mu$-a.e., then also $\lVert f_n \rVert \uparrow \lVert f \rVert $,
		\item \label{P4} $\lVert \chi_E \rVert < \infty$ for all $E \subseteq \mathcal{R}$ satisfying $\mu(E) < \infty$,
		\item \label{P5} for every $E \subseteq \mathcal{R}$ satisfying $\mu(E) < \infty$ there exists some finite constant $C_E$, dependent only on $E$, such that the inequality $ \int_E f \: d\mu \leq C_E \lVert f \rVert $ is true for all $f \in \mathcal{M}_+$.
	\end{enumerate} 
\end{definition}

\begin{definition}
	Let $\lVert \cdot \rVert : \mathcal{M}(\mathcal{R}, \mu) \rightarrow [0, \infty]$ be a mapping satisfying $\lVert \, \lvert f \rvert \, \rVert = \lVert f \rVert$ for all $f \in \mathcal{M}$. We say that $\lVert \cdot \rVert$ is a quasi-Banach function norm if its restriction to $\mathcal{M}_+$ satisfies the axioms \ref{P2}, \ref{P3} and \ref{P4} of Banach function norms together with a weaker version of axiom \ref{P1}, namely
	\begin{enumerate}[label=\textup{(Q\arabic*)}]
		\item \label{Q1} it is a quasinorm, in the sense that it satisfies the following three conditions:
		\begin{enumerate}[ref=(\theenumii)]
			\item \label{Q1a} it is positively homogeneous, i.e.\ $\forall a \in \mathbb{C} \; \forall f \in \mathcal{M}_+ : \lVert af \rVert = \lvert a \rvert \lVert f \rVert$,
			\item \label{Q1b} it satisfies  $\lVert f \rVert = 0 \Leftrightarrow f = 0$ $\mu$-a.e.,
			\item \label{Q1c} there is a constant $C\geq 1$, called the modulus of concavity of $\lVert \cdot \rVert$, such that it is subadditive up to this constant, i.e.
			\begin{equation*}
				\forall f,g \in \mathcal{M}_+ : \lVert f+g \rVert \leq C(\lVert f \rVert + \lVert g \rVert).
			\end{equation*}
		\end{enumerate}
	\end{enumerate}
\end{definition}

Usually, it is assumed that the modulus of concavity is the smallest constant for which the part \ref{Q1c} of \ref{Q1} holds. We will follow this convention, even though the value will be of little consequence for our results. 

\begin{definition}
	Let $\lVert \cdot \rVert$ be a (quasi-)Banach function norm. We say that $\lVert \cdot \rVert$ is rearrangement-invariant, abbreviated r.i., if $\lVert f\rVert = \lVert g \rVert$ whenever $f, g \in \mathcal{M}$ are equimeasurable (in the sense of Definition~\ref{DEM}).
\end{definition}

\begin{definition}
	Let $\lVert \cdot \rVert_X$ be a (quasi-)Banach function norm. We then define the corresponding (quasi-)Banach function space $X$ as the set
	\begin{equation*}
		X = \left \{ f \in \mathcal{M};  \; \lVert f \rVert_X < \infty \right \}.
	\end{equation*}
	
	Furthermore, we will say that $X$ is rearrangement-invariant whenever $\lVert \cdot \rVert_X$ is.
\end{definition}

For a detailed treatment of (r.i.) Banach function spaces we refer the reader to \cite[Chapters~1 and 2]{BennettSharpley88}; for the overview of (r.i.) quasi-Banach function spaces we recommend \cite{LoristNieraeth23}, \cite{MusilovaNekvinda24}, \cite{NekvindaPesa24}, and the references therein. Here we focus exclusively on the properties that are directly related to our work.

Firstly, an important property of r.i.~quasi-Banach function spaces over $([0, \infty), \lambda)$ is that the dilation operator is bounded on those spaces, as stated in the following theorem. This is a classical result in the context of r.i.~Banach function spaces which has been recently extended to r.i.~quasi-Banach function spaces in \cite[Section~3.4]{NekvindaPesa24} (for the classical version see e.g.~\cite[Chapter~3, Proposition~5.11]{BennettSharpley88}).

\begin{definition} \label{DDO}
	Let $t \in (0, \infty)$. The dilation operator $D_t$ is defined on $\mathcal{M}([0, \infty), \lambda)$ by the formula
	\begin{equation*}
		D_tf(s) = f(ts),
	\end{equation*}
	where $f \in \mathcal{M}([0, \infty), \lambda)$, $s \in (0, \infty)$.
\end{definition}

\begin{theorem} \label{TDRIS}
	Let $X$ be an r.i.~quasi-Banach function space over $([0, \infty), \lambda)$ and let $t \in (0, \infty)$. Then $D_t: X \rightarrow X$ is a bounded operator.
\end{theorem}

We will, however, need to use the theorem for spaces defined over measure spaces of finite measure, where the situation becomes somewhat less elegant. This is due to the fact that dilating a function changes its support. Luckily, we will only need to dilate non-increasing functions and with this additional assumption the proof contained in \cite[Section~3.4]{NekvindaPesa24} translates directly to this modified setting.

\begin{theorem} \label{TheoremDilationFinMeasure}
	Let $X$ be an r.i.~quasi-Banach function space over $([0, \alpha], \lambda)$, $\alpha < \infty$, and let $t \in (0, \infty)$. Then there is a constant $C_t \in (0, \infty)$, depending only on $X$ and $t$, such that $\lVert (D_t f) \chi_{[0, \alpha)} \rVert_X \leq C_t \lVert f \chi_{[0, \alpha)} \rVert_X$ for all non-increasing $f \in \mathcal{M}_+([0, \infty), \lambda)$.
\end{theorem}

\begin{proof}
	We only give a brief outline of the changes, as the modifications are rather straightforward. We assume that the non-increasing function $f \in \mathcal{M}_+([0, \infty), \lambda)$ is fixed from now on.
	\begin{enumerate}
		\item The monotonicity of $f$ ensures that $D_b f \leq D_a f$ $\lambda$-a.e.~for $0 < a < b < \infty$, hence the property \ref{P2} of $\lVert \cdot \rVert_X$ ensures that $\lVert (D_b f) \chi_{[0, \alpha)} \rVert_X \leq  \lVert (D_a f) \chi_{[0, \alpha)} \rVert_X$. Consequently, the conclusions of \cite[Lemmata~3.18 and 3.19]{NekvindaPesa24} remain valid.
		\item Employing the notation of \cite[Definition~3.20]{NekvindaPesa24}, we observe that $g = R_1 g + R_2 g$ for all $g \in \mathcal{M}_+([0, \infty), \lambda)$ and that \cite[Lemma~3.21]{NekvindaPesa24} remains valid (as it is a statement about a function, $\lVert \cdot \rVert_X$ plays no role).
		\item Since $ ((R_i g) \chi_{[0, \alpha)})^* \leq (R_i g)^* \chi_{[0, \alpha)}$, for $i = 1,2$ and all $g \in \mathcal{M}_+([0, \infty), \lambda)$ (which is easily verified), we may compute
		\begin{equation*}
			\begin{split}
				\lVert (D_{\frac{2}{3}} f) \chi_{[0, \alpha)} \rVert_X &\lesssim \lVert (R_1 D_{\frac{2}{3}} f) \chi_{[0, \alpha)} \rVert_X + \lVert (R_2 D_{\frac{2}{3}} f) \chi_{[0, \alpha)} \rVert_X \\
				&\leq \lVert (R_1 D_{\frac{2}{3}} f)^* \chi_{[0, \alpha)} \rVert_X + \lVert (R_2 D_{\frac{2}{3}} f)^* \chi_{[0, \alpha)} \rVert_X \\
				&\lesssim \lVert (D_{\frac{3}{2}} D_{\frac{2}{3}} f) \chi_{[0, \alpha)} \rVert_X \\
				&= \lVert f \chi_{[0, \alpha)} \rVert_X,
			\end{split}
		\end{equation*}
		where all the hidden constants are either absolute or depend only on $X$, analogously to \cite[Lemma~3.22]{NekvindaPesa24}.
		\item The extrapolation to other values of $t$ then follows precisely as in \cite[Theorem~3.23]{NekvindaPesa24}.
	\end{enumerate}
\end{proof}

Secondly, we will also work with the so called Hardy--Littlewood--P\'{o}lya relation and the related Hardy--Littlewood--P\'{o}lya principle.

\begin{definition} \label{DefHLPR}
	Let $f, g \in \mathcal{M}$. We say that $g$ majorises $f$ with respect to the Hardy--Littlewood--P\'{o}lya relation, denoted by $f \prec g$, if $f^{**} \leq g^{**}$, i.e.~if 
	\begin{equation*}
		\int_0^{t} f^* \: d\lambda \leq \int_0^{t} g^* \: d\lambda 
	\end{equation*}
	for all $t \in (0, \infty)$.
\end{definition}

\begin{definition} \label{DHLP}
	Let $\lVert \cdot \rVert$ be an r.i.~quasi-Banach function norm. We say that the Hardy--Littlewood--P\'{o}lya principle holds for $\lVert \cdot \rVert$ if the estimate $\lVert f \rVert \leq \lVert g \rVert$ is true for any pair of functions $f, g \in \mathcal{M}$ satisfying $f \prec g$.
\end{definition}

Let us put this property into context:
\begin{enumerate}
	\item The property \ref{P2} of quasi-Banach function norms states that $\lVert f \rVert \leq \lVert g \rVert$ whenever $\lvert f \rvert \leq \lvert g \rvert$ $\mu$-a.e.
	\item The property that the quasinorm is rearrangement-invariant is equivalent to saying that $\lVert f \rVert \leq \lVert g \rVert$ whenever $f^* \leq g^*$ $\lambda$-a.e. This follows from the Luxemburg representation theorem (presented as Theorem~\ref{TheoremRepresentation} below); for proof, see \cite[Corollary~3.5]{MusilovaNekvinda24}).
	\item The property that the Hardy--Littlewood--P\'{o}lya principle holds for the quasinorm states that $\lVert f \rVert \leq \lVert g \rVert$ whenever $f \prec g$.	
\end{enumerate}
Hence, we the three properties each require that the quasinorm in question is monotone with respect to some partial ordering on $\mathcal{M}_0$ and they get progressively stronger as the partial ordering gets weaker.

It is well known that the Hardy--Littlewood--P\'{o}lya principle holds for every r.i.~Banach function norm (see e.g.~\cite[Corollary~4.7]{BennettSharpley88}). The characterisation of its validity for r.i.~quasi-Banach function spaces is an interesting open problem, but for our purposes, the following necessary condition that was obtained in \cite[Lemma~2.24 and Theorem~5.9]{Pesa22} will suffice:

\begin{theorem} \label{TheoremNeccHLP}
	Let $\lVert \cdot \rVert_X$ be an r.i.~quasi-Banach function norm for which the Hardy--Littlewood--P\'{o}lya principle holds and let $X$ be the corresponding quasi-Banach function space. Then
	\begin{equation*}
		L^1 \cap L^{\infty} \hookrightarrow X \hookrightarrow L^1 + L^{\infty}.
	\end{equation*}
\end{theorem}

Here, $L^{1} \cap L^{\infty}$ and $L^1 + L^{\infty}$ are the r.i.~Banach function spaces induced by the norms
\begin{align}
	\lVert f \rVert_{L^1 + L^{\infty}} &= \int_0^{1} f^* \: d\lambda, \nonumber \\ 
	\lVert f \rVert_{L^{1} \cap L^{\infty}} &= f^*(0) + \int_0^{\infty} f^* \: d\lambda; \label{DefIntL1Infty}
\end{align}
note that said norms are equivalent to those obtained via the classical abstract constructions for sums and intersections of Banach spaces, see e.g.~\cite[Chapter~2, Section~6]{BennettSharpley88}. These formulas are beneficial to our purposes as they do not depend on the underlying measure space.

Finally, let us introduce a concept that we will briefly touch in Section~\ref{SectionRepreCorrect}, that is the concept of associate spaces. As they are not the focus of the paper, we will not go deep and limit the exposition to the strictly necessary information. To the interested reader we recommend \cite[Chapters~1 and 2]{BennettSharpley88} for an in depth treatment in the classical context of (r.i.) Banach function spaces, while a more comprehensive overview for r.i.~quasi-Banach function spaces can be found in \cite[Section~2.4]{MusilovaNekvinda24}.

\begin{definition}
	Let $\lVert \cdot \rVert_X$ be an r.i.~quasi-Banach function norm satisfying \ref{P5} and let $X$ be the corresponding quasi-Banach function space. Then the functional $\lVert \cdot \rVert_{X'}$, defined for every $f \in \mathcal{M}$ by
	\begin{equation*}
		\lVert f \rVert_{X'} = \sup_{\lVert g \rVert_X \leq 1} \int_0^{\infty} f^* g^* \: d\lambda,
	\end{equation*}
	is called the associate norm of $\lVert \cdot \rVert_X$, while the set
	\begin{equation*}
		X' = \left \{ f \in \mathcal{M}; \; \lVert f \rVert_{X'} < \infty \right \}
	\end{equation*}
	is called the associate space of $X$.
\end{definition}

The terminology is well justified, as the associate norm of an r.i.~quasi-Banach function norm satisfying \ref{P5} is always an r.i.~Banach function norm, see e.g.~\cite[Remark~2.3.(iii)]{EdmundsKerman00} or \cite[Theorem~3.1]{GogatishviliSoudsky14}.

\subsection{Absolute continuity of the quasinorm} \label{SecACqN}

\begin{definition} \label{DefACqN}
	Let $\lVert \cdot \rVert_X$ be a quasi-Banach function norm and let $X$ be the corresponding quasi-Banach function space. We say that a function $f \in X$ has absolutely continuous quasinorm if it holds that $\lVert f \chi_{E_k} \rVert_X \rightarrow 0$ as $k \to \infty$ whenever $E_k$ is a sequence of $\mu$-measurable subsets of $\mathcal{R}$ such that $\chi_{E_k} \rightarrow 0$ $\mu$-a.e.~as $k \to \infty$. The set of all such functions is denoted $X_a$.
	
	When $X_a = X$, i.e.~when every $f \in X$ has absolutely continuous quasinorm, we further say that the space $X$ itself has absolutely continuous quasinorm.
\end{definition}

We will need the following abstract version of the Lebesgue dominated convergence theorem. Its proof is precisely the same as that of \cite[Chapter~1, Proposition~3.6]{BennettSharpley88}.
\begin{proposition} \label{PropDomConv}
	Let $\lVert \cdot \rVert_X$ be a quasi-Banach function norm and let $X$ be the corresponding quasi-Banach function space. Then the following statements are equivalent:
	\begin{enumerate}
		\item $f \in X_a$.
		\item It holds for every sequence $g_n \in \mathcal{M}$ and every function $g \in \mathcal{M}$ such that $g_n \to g$ $\mu$-a.e.~and $\lvert g_n \rvert \leq \lvert f \rvert$ that all $g_n, g \in X$ and $g_n \to g$ in $X$. 
	\end{enumerate}
\end{proposition}

\subsection{Luxemburg representation theorem}

An extremely important result in the theory of r.i.~Banach function spaces it the Luxemburg representation theorem, first obtained in \cite{Luxemburg67}, which allows for a space defined over an arbitrary resonant measure space to be represented a by space defined over $[0,\mu(\mathcal{R})), \lambda)$. Recently, it has been shown that this result extends to the wider class of r.i.~quasi-Banach function spaces (see \cite[Section~3]{MusilovaNekvinda24} for details):

\begin{theorem} \label{TheoremRepresentation}
	Assume that $(\mathcal{R},\mu)$ is resonant, let $\lVert \cdot \rVert_X$ be an r.i.~quasi-Banach function norm on $\mathcal{M}(\mathcal{R},\mu)$, and denote its modulus of concavity by $C_X$. Then there is an r.i.~quasi-Banach function norm $\lVert \cdot \rVert_{\overline{X}}$ on $\mathcal{M}([0,\mu(\mathcal{R})), \lambda)$ such that for every $f \in \mathcal{M}(\mathcal{R},\mu)$ it holds that $\lVert f \rVert_X=\lVert f^* \rVert_{\overline{X}}$. Furthermore, $\lVert \cdot \rVert_{\overline{X}}$ has the property \ref{P5} whenever $\lVert \cdot \rVert_X$ does and the modulus of concavity $C_{\overline{X}}$ (of $\lVert \cdot \rVert_{\overline{X}}$) satisfies $C_{\overline{X}} \leq C_X$ when $(\mathcal{R},\mu)$ is non-atomic and $C_{\overline{X}} \leq 4C_X^2$ otherwise. Finally, when $(\mathcal{R},\mu)$ is non-atomic then $\lVert \cdot \rVert_{\overline{X}}$ is uniquely determined.
\end{theorem}

When the space is completely atomic then the representation is not unique. The reason for this is well explained in \cite[Section~3]{MusilovaNekvinda24} but it can be summarised as follow: The original quasinorm only measures the decay of a function, as the functions on atomic spaces do not have blow-ups. Since the functions on $([0,\mu(\mathcal{R})), \lambda)$ do in general have blow-ups, the representation must measure them, but how it does so cannot be uniquely determined by the original quasinorm. Hence, a choice must be made during the construction. (The authors of \cite{MusilovaNekvinda24} chose $L^1$ as the local component, because for this choice their construction when applied on an r.i.~Banach function norm yields the same representation functional as the classical construction of Luxemburg, as presented in \cite{Luxemburg67} or \cite[Chapter~2, Theorem~4.10]{BennettSharpley88}).

Similarly, the uniqueness of the non-atomic case when $\mu(\mathcal{R})< \infty$ is due to the fact that the representation space is considered as over $([0,\mu(\mathcal{R})), \lambda)$. If $([0,\infty), \lambda)$ were to be chosen as the underlying measure space (we construct such a representation space in Definition~\ref{DefRepreFixed} below), then the representation would no longer be unique for the very same reason, only with the roles of blow-ups and decay exchanged.

\section{The construction of a particular representation quasinorm} \label{SectionRepreCorrect}

From this point on, we always assume that the underlying measure space $(\mathcal{R},\mu)$ is resonant.

As is turns out, in the cases when the representation quasinorm is not uniquely determined, the representation of absolute continuity of the quasinorm may only be valid for some of the possible choices. We thus now fix the specific representation quasinorm that we will use in our results. We follow the original construction from \cite[Proof of Theorem~3.1]{MusilovaNekvinda24} with only one difference: It turns out it is more convenient for our purposes to consider representation via r.i.~quasi-Banach function norms over $([0,\infty), \lambda)$, so that they are independent of the measure of the original measure space. This is, of course, mostly a technical change as both the approaches are mostly equivalent, with the only significant difference being that the uniqueness is lost when $(\mathcal{R},\mu)$ is non-atomic and $\mu(\mathcal{R}) < \infty$.

\begin{definition} \label{DefRepreFixed}
	\mbox{}\vspace{\dimexpr-\baselineskip-\topsep} \\
	\begin{enumerate}
		\item \label{DefRepreFixed_i} Let $(\mathcal{R},\mu)$ is non-atomic and consider an arbitrary but fixed measure preserving mapping $\sigma$ from $(\mathcal{R},\mu)$ onto the range of $\mu$ (without $\infty$, if applicable; note that such a mapping always exists by \cite[Lemma~3.2]{MusilovaNekvinda24}). We then consider the operator $T :\mathcal{M}([0,\mu(\mathcal{R})), \lambda) \to \mathcal{M}(\mathcal{R},\mu)$ defined for every $f \in \mathcal{M}([0,\mu(\mathcal{R})), \lambda)$ by
		\begin{equation} \label{TheoremRepresentation:Tna}
			T(f) = f \circ \sigma,
		\end{equation}
		and define the functional $\lVert \cdot \rVert_{\overline{X_0}}$ for every $f \in \mathcal{M}([0,\mu(\mathcal{R})), \lambda)$ by the formula
		\begin{align*}
			&\lVert f \rVert_{\overline{X_0}} = \lVert T(f) \rVert_X.
		\end{align*}
		\item \label{DefRepreFixed_ii} Let $(\mathcal{R},\mu)$ be completely atomic with all atoms having the same measure which we denote by $\beta \in (0, \infty)$. As the set $\mathcal{R}$ is at most countable (given that $(\mathcal{R},\mu)$ is $\sigma$-finite), we choose some arbitrary but fixed numbering of the atoms and identify $(\mathcal{R},\mu)$ with the set $\mathcal{N} = \mathbb{N} \cap [0, \mu(\mathcal{R}))$ equipped with the measure $\nu = \beta m$, the appropriate multiple of the classical counting measure $m$. We then consider the operator $T :\mathcal{M}([0,\nu(\mathcal{N})), \lambda) \to \mathcal{M}(\mathcal{N}, \nu)$, which is defined for every $f \in \mathcal{M}([0,\nu(\mathcal{N})), \lambda)$ by the pointwise formula
		\begin{align} \label{TheoremRepresentation:T}
			T(f)(n) &= \beta^{-1} \int_{\beta n}^{\beta (n+1)} f^* \: d\lambda, & \text{for } n \in \mathcal{N}
		\end{align}
		and define the functional $\lVert \cdot \rVert_{\overline{X_0}}$ for every $f \in \mathcal{M}([0,\mu(\mathcal{R})), \lambda)$ by the formula
		\begin{equation*}
			\lVert f \rVert_{\overline{X_0}} =\left \lVert T(f) \right \rVert_X.
		\end{equation*}
	\end{enumerate}
	Finally, in both cases, we define the he functional $\lVert \cdot \rVert_{\overline{X}}$ for every $f \in \mathcal{M}([0,\infty), \lambda)$ by the formula
	\begin{equation*}
		\lVert f \rVert_{\overline{X}} = \lVert f^* \chi_{[0,\mu(\mathcal{R}))} \rVert_{\overline{X_0}}.
	\end{equation*}
\end{definition}

We would like to stress, that the mapping $\sigma$ and the ordering of atoms used in the above definition are considered to be fixed from now on, even though the resulting functional does not depend on them. This is because we will need to work with them in some of our proofs. Let us now quickly show that our modified functional $\lVert \cdot \rVert_{\overline{X}}$ has the desired properties.

\begin{proposition} \label{PropRepreFixed}
	Assume that $(\mathcal{R},\mu)$ is resonant, let $\lVert \cdot \rVert_X$ be an r.i.~quasi-Banach function norm on $\mathcal{M}(\mathcal{R},\mu)$. Then the functional  $\lVert \cdot \rVert_{\overline{X}}$ constructed in Definition~\ref{DefRepreFixed} is an r.i.~quasi-Banach function norm and satisfies for every $f \in \mathcal{M}(\mathcal{R},\mu)$ that 
	\begin{equation} \label{PropRepreFixed:1}
		\lVert f \rVert_X=\lVert f^* \rVert_{\overline{X}}.
	\end{equation}	
	Furthermore, $\lVert \cdot \rVert_{\overline{X}}$ has the property \ref{P5} if and only if $\lVert \cdot \rVert_X$ does. Finally, when $(\mathcal{R},\mu)$ is non-atomic and $\mu(\mathcal{R}) = \infty$ then $\lVert \cdot \rVert_{\overline{X}}$ is uniquely determined by \eqref{PropRepreFixed:1}.
\end{proposition}

\begin{proof}
	The functional $\lVert \cdot \rVert_{\overline{X_0}}$ is both cases precisely the functional constructed in the proof of Theorem~\ref{TheoremRepresentation} that is presented in \cite[Proof of Theorem~3.1]{MusilovaNekvinda24}, hence it has all the properties listed in said Theorem. It remains only to extend said properties to $\lVert \cdot \rVert_{\overline{X}}$.
	
	We only have to consider the case $\mu(\mathcal{R}) < \infty$ as otherwise there is nothing to prove. The properties \ref{P2}, \ref{P3} and \ref{P4} as well as parts \ref{Q1a} and \ref{Q1b} of the axiom \ref{Q1} are easy consequences of the respective properties of $\lVert \cdot \rVert_{\overline{X_0}}$ and the properties of non-increasing rearrangement. Furthermore, the rearrangement invariance is obvious. To prove the remaining properties we employ a combination of the tools presented in the proofs of \cite[Proposition~3.3 and Theorem~3.4]{Pesa22}.
	
	As for the quasi-triangle inequality, it follows from the properties of non-increasing rearrangement (see e.g.~\cite[Chapter~2, Proposition~1.7]{BennettSharpley88}) and Theorem~\ref{TheoremDilationFinMeasure}, as they imply that
	\begin{align*}
		\lVert f+g \rVert_{\overline{X}} &= \lVert (f+g)^* \chi_{[0,\mu(\mathcal{R}))} \rVert_{\overline{X_0}} \\
			&\leq \lVert (D_{\frac{1}{2}} f^* + D_{\frac{1}{2}} g^*) \chi_{[0,\mu(\mathcal{R}))} \rVert_{\overline{X_0}} \\ 
			&\lesssim \lVert D_{\frac{1}{2}} f^* \chi_{[0,\mu(\mathcal{R}))} \rVert_{\overline{X_0}} + \lVert D_{\frac{1}{2}} g^* \chi_{[0,\mu(\mathcal{R}))} \rVert_{\overline{X_0}} \\
			&\lesssim \lVert f^* \chi_{[0,\mu(\mathcal{R}))} \rVert_{\overline{X_0}} + \lVert g^* \chi_{[0,\mu(\mathcal{R}))} \rVert_{\overline{X_0}}.
	\end{align*}
	
	As for the part concerning the property \ref{P5}, we start by the sufficiency and fix some set $E \subseteq [0, \infty)$ of finite measure. We may, without loss of generality, assume that $\lambda(E) > \mu(\mathcal{R})$, because otherwise the proof is similar but simpler. Then, by Hardy--Littlewood inequality (Theorem~\ref{THLI}), it holds for every $f \in \mathcal{M}_+([0,\infty), \lambda)$ that
	\begin{equation} \label{PropRepreFixed:2}
		\begin{split}
			\int_{E} f \: d\lambda &\leq \int_{0}^{\lambda(E)} f^* \: d\lambda \\
			&= \int_0^{\mu(\mathcal{R})} f^* \: d\lambda + \int_{\mu(\mathcal{R})}^{\lambda(E)} f^* \: d\lambda \\
			&\leq \int_0^{\mu(\mathcal{R})} f^* \: d\lambda + (\lambda(E) - \mu(\mathcal{R})) f^*(\mu(\mathcal{R})) \\
			&\leq \frac{\lambda(E)}{\mu(\mathcal{R})} \int_0^{\mu(\mathcal{R})} f^* \: d\lambda \\
			&\leq \frac{\lambda(E)}{\mu(\mathcal{R})} C_{[0,\mu(\mathcal{R}))} \lVert f^*\chi_{[0,\mu(\mathcal{R}))} \rVert_{\overline{X_0}},
		\end{split}
	\end{equation}
	where $C_{[0,\mu(\mathcal{R}))}$ is the constant from the property \ref{P5} of $\lVert \cdot \rVert_{\overline{X_0}}$ for the set $[0,\mu(\mathcal{R}))$ (which exist by our assumptions that $\mu(\mathcal{R})< \lambda(E) <\infty$ and that $\lVert \cdot \rVert_X$ has the property \ref{P5}).
	
	Finally, the necessity is simpler, as it follows from the Hardy--Littlewood inequality (Theorem~\ref{THLI}) that we have for any $g \in \mathcal{M}_+(\mathcal{R}, \mu)$ and any $E \subseteq \mathcal{R}$ with $\mu(E) < \infty$ that
	\begin{equation*}
		\int_E g \: d\mu \leq \int_0^{\mu(E)} g^* \: d\lambda \leq  C_{[0,\mu(E))} \lVert g^* \rVert_{\overline{X}} = C_{[0,\mu(E))} \lVert g^* \rVert_{\overline{X}_0} = C_{[0,\mu(E))} \lVert g \rVert_{X},
	\end{equation*}
	where $C_{[0,\mu(E))}$ is the constant from the property \ref{P5} of $\lVert \cdot \rVert_{\overline{X}}$ and the second to last equality employs the fact that $g^* = g^*\chi_{[0,\mu(\mathcal{R}))}$.
\end{proof}

Observe, that when $\mu(\mathcal{R}) < \infty$, we have for every $f \in \mathcal{M}([0, \infty), \lambda)$ that
\begin{equation*}
	 \lVert f^* \chi_{(\mu(\mathcal{R}), \infty)} \rVert_{L^{\infty}} = f^*(\mu(\mathcal{R})) \leq \frac{1}{\lVert \chi_{[0, \mu(\mathcal{R}))} \rVert_{\overline{X_0}}} \lVert f^* \chi_{[0, \mu(\mathcal{R}))} \rVert_{\overline{X_0}},
\end{equation*}
which shows that our approach to constructing a functional over $([0,\infty), \lambda)$ from a given functional over $([0,\mu(\mathcal{R})), \lambda)$ can in essence be described as constructing a kind of an r.i.~amalgam space where the original functional serves as the local component while the global component is chosen to be $L^{\infty}$. (Similar observation has been made in \cite[Proposition~3.3]{Pesa22}, the proof of which includes an argument similar to the one we use in \eqref{PropRepreFixed:2}.) Hence, our construction is conceptually similar to the approach used in \cite[Proof of Theorem~3.1]{MusilovaNekvinda24} to construct the representation quasinorm for the case of completely atomic measure. Our motivation for choosing $L^{\infty}$ is also similar: for this choice, the construction applied on r.i.~Banach function spaces yields the same representation space as the classical approach of Luxemburg, as presented in \cite{Luxemburg67} or \cite[Chapter~2, Theorem~4.10]{BennettSharpley88}.

Next, we observe that our choice of representation functional is well compatible with the Hardy--Littlewood--P{\' o}lya principle. Besides the obvious practical applications, this statement also serves as a further argument in favour of choosing $L^1$ as the local component of the representation functional in the case when $(\mathcal{R},\mu)$ is completely atomic.

\begin{proposition} \label{PropRepreHLP}
Assume that $(\mathcal{R},\mu)$ is resonant, let $\lVert \cdot \rVert_X$ be an r.i.~quasi-Banach function norm on $\mathcal{M}(\mathcal{R},\mu)$, and let $X$ be the corresponding r.i.~quasi-Banach function space. Let further $\lVert \cdot \rVert_{\overline{X}}$ be the r.i.~quasi-Banach function norm constructed in Definition~\ref{DefRepreFixed} and let $\overline{X}$ be the corresponding r.i.~quasi-Banach function space. Then  the Hardy--Littlewood--P{\' o}lya principle holds for $\lVert \cdot \rVert_X$ if and only if it holds for $\lVert \cdot \rVert_{\overline{X}}$.
\end{proposition}

\begin{proof}	
	The sufficiency is obvious, as clearly $f \prec g$ (with $f, g \in \mathcal{M}(\mathcal{R},\mu)$) if and only if $f^* \prec g^*$.
	
	Similarly, when $(\mathcal{R},\mu)$ is non-atomic and $T$ is as in part \ref{DefRepreFixed_i} of Definition~\ref{DefRepreFixed}, we have for all $h \in \mathcal{M}([0,\infty), \lambda)$ that $T(h^* \chi_{[0,\mu(\mathcal{R}))})^* = h^* \chi_{[0,\mu(\mathcal{R}))}$, from which the necessity follows.
	
	Finally, in the case when $(\mathcal{R},\mu)$ is completely atomic and $T$ is as in part \ref{DefRepreFixed_ii}, the argument for necessity is slightly more complicated.
	
	Let $h \in \overline{X}$ and $t \in (0, \infty)$ and put
	\begin{align*}
		t_0 &= \min\{t, \mu(\mathcal{R})\} \\
		n_t &= \sup \{ n \in \mathbb{N} \cup \{ -1 \}; \; (n+1) \beta \leq t_0\}, \\
		\mathcal{N}_t &= \{n \in \mathbb{N}; \; n \leq n_t\}.
	\end{align*}
	Then we have
	\begin{equation} \label{PropRepreHLP:E1}
		\int_{0}^{t} T(h^* \chi_{[0,\mu(\mathcal{R}))})^* = \sum_{n \in \mathcal{N}_t} \int_{n \beta}^{(n+1) \beta} h^* \: d\lambda + \frac{t_0 - (n_t+1) \beta}{\beta} \int_{(n_t +1) \beta}^{(n_t+2) \beta} h^* \: d\lambda.
	\end{equation}
	Note that in the case when $(n_t +1) \beta = \mu(\mathcal{R})$ the last integral ``sees values outside of $[0,\mu(\mathcal{R}))$'', but this is not a problem as the integral is in this case multiplied by zero. When $t$ is in the range of $\mu$, then \eqref{PropRepreHLP:E1} further simplifies to
	\begin{equation} \label{PropRepreHLP:E2}
		\int_{0}^{t} T(h^* \chi_{[0,\mu(\mathcal{R}))})^* = \sum_{n \in \mathcal{N}_t} \int_{n \beta}^{(n+1) \beta} h^* \: d\lambda = \int_0^{t} h^* \: d\lambda.
	\end{equation}
	
	Consider now some functions $f,g \in \overline{X}$ satisfying $f \prec g$. It follows from \eqref{PropRepreHLP:E2} that
	\begin{align} \label{PropRepreHLP:E3}
		\int_{0}^{t} T(f^* \chi_{[0,\mu(\mathcal{R}))})^* \: d\lambda &\leq \int_{0}^{t} T(g^* \chi_{[0,\mu(\mathcal{R}))})^* \: d\lambda &\text{for } t \text{ in the range of } \mu.
	\end{align}
	Moreover, it follows from \eqref{PropRepreHLP:E1} that the integrals we are comparing in \eqref{PropRepreHLP:E3}, when considered as functions of $t$, are continuous; linear on the intervals $[n \beta, (n+1) \beta]$, where $(n+1) \beta \leq \mu(\mathcal{R})$; and constant on $[\mu(\mathcal{R}), \infty)$ (when $\mu(\mathcal{R}) < \infty$). Hence, the same estimate as in \eqref{PropRepreHLP:E3} also holds for all other $t \in (0, \infty)$ and the necessity follows.
\end{proof}

Note that the proof of necessity in the case of completely atomic measure space depends heavily on our choice of $L^1$ as the local component of the representation quasinorm, as this choice is what allows us to perform the computation in \eqref{PropRepreHLP:E2}.

Finally, let us show that our construction is well compatible with the concept of associate spaces. This result is slightly tangential to our purposes, but it does provide some understanding of the context surrounding Theorem~\ref{TheoremRepreACqN}. We also believe that it is interesting in its own right and that it should be included to make the exposition complete.

\begin{proposition} \label{PropRepreAS}
	Assume that $(\mathcal{R},\mu)$ is resonant, let $\lVert \cdot \rVert_X$ be an r.i.~quasi-Banach function norm on $\mathcal{M}(\mathcal{R},\mu)$, and let $X$ be the corresponding r.i.~quasi-Banach function space. Let further $\lVert \cdot \rVert_{\overline{X}}$ be the r.i.~quasi-Banach function norm constructed in Definition~\ref{DefRepreFixed} and let $\overline{X}$ be the corresponding r.i.~quasi-Banach function space. Finally, denote by $\lVert \cdot \rVert_{X'}$ and $\lVert \cdot \rVert_{\left ( \overline{X} \right )'}$ the respective associate norms (on $\mathcal{M}(\mathcal{R},\mu)$ and $\mathcal{M}([0,\infty), \lambda)$, respectively). Then it holds for every $f \in \mathcal{M}(\mathcal{R},\mu)$ that
	\begin{equation} \label{PropRepreAS:E1}
		\lVert f \rVert_{X'} = \lVert f^* \rVert_{\left ( \overline{X} \right )'}.
	\end{equation}
\end{proposition}

We stress that we do not claim that $\lVert \cdot \rVert_{\left ( \overline{X} \right )'}$ is the same functional as $\lVert \cdot \rVert_{\overline{X'}}$. Indeed, such an equality does not hold in general, that is, it holds if and only if $(\mathcal{R},\mu)$ is non-atomic and of infinite measure. This is because in the case of atomic measure, the local component of $\overline{X}$ is $L^1$ and thus the local component of its associate space is $(L^1)' = L^{\infty}$ (this can be shown in a manner similar to the proof of \cite[Theorem~3.5]{Pesa22}). Similarly, when $\mu(\mathcal{R}) < \infty$, then the global component of $\overline{X}$ is $L^{\infty}$ and thus the global component of its associate space is $L^1$. Hence, in either case, $\lVert \cdot \rVert_{\left ( \overline{X} \right )'}$ is a representation of $\lVert \cdot \rVert_{X'}$ that is distinct of the one constructed in Definition~\ref{DefRepreFixed}.

Let us also note that in the context of r.i.~Banach function spaces a more abstract version of this statement is known, see e.g.~\cite[Chapter~2, Theorem~4.10]{BennettSharpley88}.

\begin{proof}[Proof of Proposition~\ref{PropRepreAS}]
	It is clear from the definition that we have for every fixed $f \in \mathcal{M}(\mathcal{R},\mu)$
	\begin{equation*}
		\lVert f \rVert_{X'} \leq \lVert f^* \rVert_{\left ( \overline{X} \right )'},
	\end{equation*}
	as the supremum defining the functional on the right-hand side considers (in general) a larger set of functions than that at the left-hand side. Proving the inverse inequality will, as usual, require us to distinguish the two cases of resonant measure spaces.
	
	When $(\mathcal{R},\mu)$ is non-atomic, then we put $T$ is as in part \ref{DefRepreFixed_i} of Definition~\ref{DefRepreFixed} and observe that it holds for every $g \in \overline{X}$, $\lVert g \rVert_{\overline{X}} \leq 1$ that $\lVert T(g^* \chi_{[0, \mu(\mathcal{R}))}) \rVert_X = \lVert g \rVert_{\overline{X}} \leq 1$ while $(T(g^* \chi_{[0, \mu(\mathcal{R}))}))^* = g^* \chi_{[0, \mu(\mathcal{R}))}$ and thus
	\begin{equation*}
		\int_0^{\infty} f^* g^* \: d\lambda = \int_0^{\mu(\mathcal{R})} f^* g^* \: d\lambda = \int_0^{\infty} f^* (T(g^* \chi_{[0, \mu(\mathcal{R}))}))^* \: d\lambda.
	\end{equation*}
	This establishes \eqref{PropRepreAS:E1}.
	
	The case of completely atomic measure is similar. We put $T$ is as in part \ref{DefRepreFixed_ii} of Definition~\ref{DefRepreFixed} and observe that it again holds for every $g \in \overline{X}$, $\lVert g \rVert_{\overline{X}} \leq 1$ that $\lVert T(g^* \chi_{[0, \mu(\mathcal{R}))}) \rVert_X = \lVert g \rVert_{\overline{X}} \leq 1$. The difference is that in this case the non-increasing rearrangement of $T(g^* \chi_{[0, \mu(\mathcal{R}))})$ does not have to coincide with $g^* \chi_{[0, \mu(\mathcal{R}))}$. However, since $f^*$ is constant on the intervals $[n \beta, (n+1) \beta)$, where $n \in \mathbb{N}$, $(n+1) \beta \leq \mu(\mathcal{R})$, we may still compute
	\begin{equation*}
		\begin{split}
			\int_0^{\infty} f^* g^* \: d\lambda &= \int_0^{\mu(\mathcal{R})} f^* g^* \: d\lambda \\
			&= \sum_{\substack{n \in \mathbb{N} \\ (n+1) \beta \leq \mu(\mathcal{R})}} f^*(n) \int_{\beta n}^{\beta (n+1)} g^* \: d\lambda \\
			&= \sum_{\substack{n \in \mathbb{N} \\ (n+1) \beta \leq \mu(\mathcal{R})}} \beta f^*(n) (T(g^* \chi_{[0, \mu(\mathcal{R}))}))^*(n) \\
			&= \int_0^{\infty} f^* (T(g^* \chi_{[0, \mu(\mathcal{R}))}))^* \: d\lambda,
		\end{split}
	\end{equation*}
	which again establishes \eqref{PropRepreAS:E1}.
\end{proof}

\section{Representation of the absolute continuity of the quasinorm} \label{SectionRepreACqN}

We continue working under the assumption that the underlying measure space $(\mathcal{R},\mu)$ is resonant.

\begin{proposition} \label{PropACNimpliesACR}
	Let $\lVert \cdot \rVert_X$ be an r.i.~quasi-Banach function norm and let $X$ be the corresponding r.i.~quasi-Banach function space. If $f \in X$ has an absolutely continuous quasinorm (i.e.~$f \in X_a$), then it satisfies
	\begin{equation*}
		\lim_{t \to \infty} f^*(t) = 0.
	\end{equation*}
\end{proposition}

\begin{proof}
	Assume contrary, then there is some $\alpha \in (0, \infty)$ and $F \subseteq \mathcal{R}$ such that $\mu(F) = \infty$ and $\lvert f \rvert > \alpha \chi_F$. 
	
	Consider now the sequence of sets $R_k \subseteq \mathcal{R}$ such that $\mu(\mathcal{R}_k) < \infty$, $R_k \subseteq R_{k+1}$, and $\bigcup_{k = 0}^{\infty} R_k = \mathcal{R}$. The sets $E_k = \mathcal{R} \setminus R_k$ satisfy $\chi_{E_k} \to 0$ $\mu$-a.e.~as $k \to \infty$, but clearly $\mu(F \cap E_k) = \infty$ and thus the rearrangement invariance of $\lVert \cdot \rVert_X$ implies (together with its properties \ref{P1} and \ref{P2})
	\begin{equation*}
		\lVert f \chi_{E_k} \rVert_X \geq \alpha \lVert \chi_F \chi_{E_k} \rVert_X = \alpha \lVert \chi_F \rVert_X > 0.
	\end{equation*}
\end{proof}

\begin{theorem} \label{TheoremRepreACqN}
	Let $\lVert \cdot \rVert_X$ be an r.i.~quasi-Banach function norm and let $X$ be the corresponding r.i.~quasi-Banach function space. Let $\lVert \cdot \rVert_{\overline{X}}$ be the r.i.~quasi-Banach function norm constructed in Definition~\ref{DefRepreFixed} and let $\overline{X}$ be the corresponding quasi-Banach function space. Then it holds for every function $f \in X$ that $f \in X_a$ if and only if $f^* \in \left ( \overline{X} \right )_a$.
\end{theorem}

\begin{proof}
	It follows from Proposition~\ref{PropACNimpliesACR} that if either side of the desired equivalence holds, then 
	\begin{equation} \label{TheoremRepreACqN:E0}
		\lim_{t \to \infty} f^*(t) = 0.
	\end{equation}
	
	Assume first that $f^* \in \left ( \overline{X} \right )_a$ and that we have a sequence of sets $E_k \subseteq \mathcal{R}$ such that $\chi_{E_k} \to 0$ $\mu$-a.e. Since $(f\chi_{E_k})^* \leq f^*$, \eqref{TheoremRepreACqN:E0} and the Propositions~\ref{PropACR} and \ref{PropDomConv} together imply that
	\begin{align*}
		\lVert f\chi_{E_k} \rVert_X &= \lVert (f\chi_{E_k})^* \rVert_{\overline{X}} \to 0 & \text{as } k \to \infty.
	\end{align*}
	
	Assume now that $f \in X_a$. We need perform the proof separately for the two cases of resonant measure spaces, i.e.~non-atomic spaces and completely atomic spaces.
	
	We first consider the case of non-atomic measure. By the means of \cite[Chapter~2, Corollary~7.6]{BennettSharpley88}, we obtain from \eqref{TheoremRepreACqN:E0} that there is a measure preserving mapping $\sigma_f$ from the support of $f$ onto the support of $f^*$ such that $\lvert f \rvert = f^* \circ \sigma_f$ on the support of $f$. We define an operator $T_f :\mathcal{M}([0,\infty), \lambda) \to \mathcal{M}(\mathcal{R},\mu)$ defined for every $g \in \mathcal{M}([0,\infty), \lambda)$ by
	\begin{align*}
		T_f(g) &= 
		\begin{cases}
			g \circ \sigma_f &\text{on } \supp f, \\
			0 &\text{elsewhere.}
		\end{cases} 
	\end{align*}
	Since $\sigma_f$ is measure-preserving, it is clear that it holds for any function $g \in \mathcal{M}([0,\infty), \lambda)$ satisfying $\supp g \subseteq \supp f^*$ that $(T_f(g))^* = g^*$ and consequently
	\begin{align} \label{TheoremRepreACqN:E1}
		\lVert g \rVert_{\overline{X}} = \lVert T_f(g) \rVert_X. 
	\end{align}
	
	Now, consider some sequence of sets $E_k \subseteq [0, \infty)$ such that $\chi_{E_k} \rightarrow 0$ $\lambda$-a.e. Clearly, the sets $\widetilde{E_k} = \sigma_f^{-1} (E_k \cap \supp f^*)$ satisfy $\chi_{\widetilde{E_k}} \rightarrow 0$ $\mu$-a.e. Hence, we may apply \eqref{TheoremRepreACqN:E1} together with our assumption $f \in X_a$ to compute
	\begin{align*}
		\lVert f^* \chi_{E_k} \rVert_{\overline{X}} &= \lVert f^* \chi_{E_k\cap \supp f^*} \rVert_{\overline{X}} = \lVert T_f(f^* \chi_{E_k\cap \supp f^*}) \rVert_X = \lVert \lvert f \rvert \chi_{\widetilde{E_k}} \rVert_X \to 0 &\text{as } n \to \infty.
	\end{align*}
	
	Consider now the case when $(\mathcal{R}, \mu)$ is completely atomic with all atoms having the same measure (denoted $\beta \in (0, \infty)$). It follows from \eqref{TheoremRepreACqN:E0} that there exists some enumeration of the atoms, denoted $\widetilde{e_n}$, for which it holds that 
	\begin{equation*}
		\lvert f \rvert (\widetilde{e_n}) = f^*(\beta n).
	\end{equation*}
	We stress that this numbering of atoms depends on $f$. 
	
	We now consider the operator $T_f :\mathcal{M}([0, \infty), \lambda) \to \mathcal{M}(\mathcal{R}, \mu)$, which is defined for every $g \in \mathcal{M}([0,\infty), \lambda)$ by the pointwise formula
	\begin{align} \label{TheoremRepreACqN:E2}
		T_f(g)(\widetilde{e_n}) &= \beta^{-1} \int_{\beta n}^{\beta (n+1)} g^* \: d\lambda, & \text{for } n \in \mathbb{N} \cap [0, \mu(\mathcal{R})).
	\end{align}
	Clearly, the difference between the operators $T_f$ defined in \eqref{TheoremRepreACqN:E2} and the operator $T$ defined in \eqref{TheoremRepresentation:T} is precisely that (in general) each uses a different enumeration of the atoms. Hence, it is not difficult to verify that we have for every $g \in \mathcal{M}([0,\infty), \lambda)$ that $(T_f(g))^* = (T(g))^*$ and thus
	\begin{equation} \label{TheoremRepreACqN:E3}
		\lVert T_f(g) \rVert_X = \lVert T(g) \rVert_X = \lVert g \rVert_{\overline{X}}.
	\end{equation}
	
	Let us now fix some sequence of sets $E_k \subseteq [0, \infty)$ such that $\chi_{E_k} \rightarrow 0$ $\lambda$-a.e. Then we have by \eqref{TheoremRepreACqN:E0} and Proposition~\ref{PropACR} that 
	\begin{align*}
		\lim_{k \to \infty} (f^* \chi_{E_k})^*(t) &= 0 & \text{for } t \in (0, \infty).
	\end{align*}
	Consequently, the Lebesgue dominated convergence theorem implies for every $n \in \mathbb{N} \cap [0, \mu(\mathcal{R}))$ that
	\begin{align} \label{TheoremRepreACqN:E4}
		\beta^{-1} \int_{\beta n}^{\beta (n+1)}  (f^* \chi_{E_k})^* \: d\lambda &\to 0 &\textup{as } k \to \infty.
	\end{align}
	Here, the functions $f^*(\beta n) \chi_{[\beta n, \beta (n+1))}$ serve as the respective majorants (recall that $f^*$ is constant on the relevant intervals). We have thus shown that
	\begin{align*}
		\lvert f \rvert \geq T_f (f^* \chi_{E_k}) &\to 0 \text{ $\mu$-a.e.} &\textup{as } k \to \infty.
	\end{align*} 
	It remains to combine Proposition~\ref{PropDomConv} with \eqref{TheoremRepreACqN:E3} to compute
	\begin{align*}
		\lVert f^* \chi_{E_k} \rVert_{\overline{X}} = \lVert T_f (f^* \chi_{E_k}) \rVert_X &\to 0 &\textup{as } k \to \infty.
	\end{align*}
\end{proof}

Let us now justify our claim that the validity of the theorem in the cases when the representation functional is not uniquely determined depends on the precise choice that is considered. As can be seen from the proof above, problems can only arise in relation to the necessity part of the statement (i.e.~$f \in X_a \implies f^* \in \left ( \overline{X} \right )_a$) in the case when the underlying measure spaces is completely atomic. Indeed, the proof of sufficiency is valid for every possible choice of representation, while in the case of necessity for non-atomic finite measure, the choice of the global component of the representation quasinorm plays no role as none of its properties are used. This corresponds to the fact, that the non-increasing rearrangement of a function defined on $(\mathcal{R}, \mu)$ is constant zero on the interval $(\mu(\mathcal{R}), \infty)$. 

On the other hand, in the case of atomic measure, the non-increasing rearrangements of functions defined on  $(\mathcal{R}, \mu)$ are constant on $[0, 1)$, which is not enough to guarantee absolute continuity of the quasinorm for every possible choice of the local component of the representation quasinorm. As the construction in Definition~\ref{DefRepreFixed} employs $L^1$ for this role, we may obtain \eqref{TheoremRepreACqN:E4} via the Lebesgue dominated convergence theorem. In fact, this step could be performed for any choice of the local component for which simple functions have absolutely continuous quasinorm, that is, for every r.i.~quasi-Banach function space other than $L^{\infty}$ (see \cite[Theorem~4.15]{MusilovaNekvinda24} for the precise meaning and the justification of this last claim). However, this sole counterexample is rather significant for the following reason: when $\lVert \cdot \rVert_X$ satisfies \ref{P5}, then we have shown in Proposition~\ref{PropRepreAS} that $\left ( \overline{X} \right )'$, the associate space of $\overline{X}$, is a possible representation space of $X'$, the associate space of $X$ (see also \cite[Chapter~2, Theorem~4.10]{BennettSharpley88} for the classical case of r.i.~Banach function spaces). As the local component of $\overline{X}$ is $L^1$, the local component of its associate space is $(L^1)' = L^{\infty}$ (this can be shown in a manner similar to the proof of \cite[Theorem~3.5]{Pesa22}). Hence, this approach leads naturally to a representation of $X'$ for which Theorem~\ref{TheoremRepreACqN} does not hold.

An immediate consequence of Theorem~\ref{TheoremRepreACqN} is that for r.i.~quasi-Banach function space $X$, the set $X_a$ is an order ideal with respect to the pointwise ordering of rearrangements.

\begin{corollary} \label{CorollaryACRI}
	Let $\lVert \cdot \rVert_X$ be an r.i.~quasi-Banach function norm and let $X$ be the corresponding r.i.~quasi-Banach function space. Assume that $f, g \in X$ satisfy
	$f^* \leq g^*$ $\lambda$-a.e.~and $g \in X_a$. Then also $f \in X_a$.
\end{corollary}

\begin{proof}
	One only has to combine Theorem~\ref{TheoremRepreACqN} with the property \ref{P2} of $\lVert \cdot \rVert_{\overline{X}}$.
\end{proof}

A similar statement also holds for the Hardy--Littlewood--P\'{o}lya relation for those r.i.~quasi-Banach function norms for which the Hardy--Littlewood--P\'{o}lya principle holds. The proof is more interesting, however.

\begin{theorem} \label{TheoremACHLP}
	Let $\lVert \cdot \rVert_X$ be an r.i.~quasi-Banach function norm for which the Hardy--Littlewood--P\'{o}lya principle holds and let $X$ be the corresponding r.i.~quasi-Banach function space. Assume that $f, g \in X$ satisfy
	$f \prec g$ and $g \in X_a$. Then also $f \in X_a$.
\end{theorem}

To prove Theorem~\ref{TheoremACHLP}, we first need two statements that are of independent interest. The first one shows that in order to prove absolute continuity of a given function, one only needs to consider two types of sequences of the sets $E_k \subseteq [0,\infty)$ (applied onto the non-increasing rearrangement).

\begin{proposition} \label{PropAmalAC}
	Let $\lVert \cdot \rVert_X$ be an r.i.~quasi-Banach function norm and let $X$ be the corresponding r.i.~quasi-Banach function space. Let further $\lVert \cdot \rVert_{\overline{X}}$ be the r.i.~quasi-Banach function norm constructed in Definition~\ref{DefRepreFixed} and let $\overline{X}$ be the corresponding r.i.~quasi-Banach function space. Then $f \in X$ has an absolutely continuous quasinorm (i.e.~$f \in X_a$) if and only if it satisfies both
	\begin{align}
		\lim_{k \to \infty} \lVert f^* \chi_{[0,k^{-1})} \rVert_{\overline{X}} &= 0, \label{PropAmalAC:Eloc}\\
		\lim_{k \to \infty} \lVert f^* \chi_{[k, \infty)} \rVert_{\overline{X}} &= 0. \label{PropAmalAC:Eglob}
	\end{align}
\end{proposition}

\begin{proof}
	The necessity follows directly from Theorem~\ref{TheoremRepreACqN}. The same theorem also implies that the sufficiency will follow once we show that \eqref{PropAmalAC:Eloc} and \eqref{PropAmalAC:Eglob} together imply $f^* \in \left ( \overline{X} \right )_a$.
	
	Let $E_k \subseteq [0, \infty)$ be arbitrary such that $\chi_{E_k} \to 0$ $\lambda$-a.e.~and fix $\varepsilon > 0$. By \eqref{PropAmalAC:Eglob}, there is a $k_0$ such that 
	\begin{equation} \label{PropAmalAC:E1}
		\lVert f^* \chi_{[k_0, \infty)} \rVert_{\overline{X}} < \varepsilon.
	\end{equation}
	Put $\widetilde{E_k} = E_k \cap [0, k_0]$. Then $\lambda(\widetilde{E_k}) < \infty$ and thus the continuity of measure implies $\lambda(\widetilde{E_k}) \to 0$ as $k\to \infty$. Hence, using \eqref{PropAmalAC:Eloc} to find $k_1$ such that
	\begin{equation*}
		\lVert f^* \chi_{[0, k_1^{-1})} \rVert_{\overline{X}} < \varepsilon
	\end{equation*}
	and then finding $k_2$ such that for every $k \geq k_2$ we have $\lambda(\widetilde{E_k}) \leq k_1^{-1}$, we obtain from the rearrangement-invariance of $\lVert \cdot \rVert_{\overline{X}}$ and its property \ref{P2} that it holds for every such $k$ that	
	\begin{equation} \label{PropAmalAC:E2}
		\lVert f^* \chi_{\widetilde{E_k}} \rVert_{\overline{X}} = \lVert (f^* \chi_{\widetilde{E_k}})^* \rVert_{\overline{X}} \leq \lVert f^* \chi_{[0,k_1^{-1})} \rVert_{\overline{X}} < \varepsilon.
	\end{equation}
	Finally, by combining \eqref{PropAmalAC:E1} and \eqref{PropAmalAC:E2} with the properties \ref{Q1} and \ref{P2} of $\lVert \cdot \rVert_{\overline{X}}$, we get
	\begin{equation*}
		\lVert f^* \chi_{{E_k}} \rVert_{\overline{X}} \leq C_{\overline{X}} (\lVert f^* \chi_{\widetilde{E_k}} \rVert_{\overline{X}} + \lVert f^* \chi_{[k_0, \infty)} \rVert_{\overline{X}}) \leq 2C_{\overline{X}} \varepsilon
	\end{equation*}
	for every $k \geq k_2$.
\end{proof}

Let us note that when $(\mathcal{R}, \mu)$ is completely atomic, then
\begin{enumerate}
	\item the necessity of \eqref{PropAmalAC:Eloc} depends on the proper choice of the representation functional (by the same reasoning as in Theorem~\ref{TheoremRepreACqN}),
	\item for the functional $\lVert \cdot \rVert_{\overline{X}}$ we have constructed in Definition~\ref{DefRepreFixed} the condition \eqref{PropAmalAC:Eloc} holds trivially.
\end{enumerate}
Further, when $\mu(\mathcal{R}) < \infty$ then \eqref{PropAmalAC:Eglob} holds trivially.

We now move to the second statement we require, which shows that for r.i.~quasi-Banach function spaces $X$ that are sufficiently large the subspace $X_a$ contains the space $L^1 \cap L^{\infty}$.

\begin{proposition} \label{PropIncluCap}
	Let $\lVert \cdot \rVert_X$ be an r.i.~quasi-Banach function norm and let $X$ be the corresponding r.i.~quasi-Banach function space. Let further $\lVert \cdot \rVert_{\overline{X}}$ be the r.i.~quasi-Banach function norm constructed in Definition~\ref{DefRepreFixed} and let $\overline{X}$ be the corresponding r.i.~quasi-Banach function space. Assume that $L^1 \cap L^{\infty} \hookrightarrow X$ and that there is some $f_0 \in \overline{X}$ for which $f_0^*(0) = \infty$. Then $L^1 \cap L^{\infty} \subseteq X_a$.
\end{proposition}

In the terminology of \cite{Pesa22} (and in light of the results presented there), the assumptions of the proposition can be formulated as the global component of $X$ being weaker that that of $L^1$ while its local component is supposed to be strictly weaker than that of $L^{\infty}$. Note also that the latter assumption is always satisfied when $(\mathcal{R}, \mu)$ is completely atomic and also that the embedding $L^1 \cap L^{\infty} \hookrightarrow X$ is in fact equivalent to the set-theoretical inclusion $L^1 \cap L^{\infty} \subseteq X$ (see \cite[Corollary~3.10]{NekvindaPesa24}).

\begin{proof}	
	We will use Proposition~\ref{PropAmalAC} and show in turn that \eqref{PropAmalAC:Eloc} and \eqref{PropAmalAC:Eglob} hold for every $f \in L^1 \cap L^{\infty}$ (which is fixed from now on; we also stress that $f \in \mathcal{M}(\mathcal{R}, \mu)$).
		
	We start with \eqref{PropAmalAC:Eloc} which is simpler, as it follows from \cite[Theorem~4.15]{MusilovaNekvinda24} (applied on $\overline{X}$) that the existence of $f_0 \in \overline{X}$ for which $f_0^*(0) = \infty$ implies that
	\begin{equation*}
		\lim_{k \to \infty} \lVert f^* \chi_{[0,k^{-1})} \rVert_{\overline{X}} \leq  f^*(0)\lim_{k \to \infty} \lVert \chi_{[0,k^{-1})} \rVert_{\overline{X}} = 0.
	\end{equation*}
	
	As for \eqref{PropAmalAC:Eglob}, the first step is to observe that if follows from our assumption $L^1 \cap L^{\infty} \hookrightarrow X$ that we have for every $g \in \mathcal{M}(\mathcal{R}, \mu)$ that
	\begin{equation} \label{PropIncluCap:E1}
		\lVert g^* \rVert_{\overline{X}} = \lVert g \rVert_{X} \lesssim \lVert g \rVert_{L^1 \cap L^{\infty}} \approx \max \left \{ g^*(0), \; \lVert g^* \rVert_{L^1} \right \}, 
	\end{equation}
	where the norm $\lVert \cdot \rVert_{L^1}$ is considered over $([0, \infty), \lambda)$ (we use the equivalent definition of $\lVert \cdot \rVert_{L^1 \cap L^{\infty}}$ presented as \eqref{DefIntL1Infty}). 
	
	We now define the number $\beta \in (0, \infty)$ by
	\begin{equation*}
		\beta = \begin{cases}
			\mu(e) & \text{when $(\mathcal{R}, \mu)$ is completely atomic and $e$ is an arbitrary atom,} \\
			1 & \text{when $(\mathcal{R}, \mu)$ is non-atomic.}
		\end{cases}
	\end{equation*}
	Note that $\beta$ is well defined (since we assume that $(\mathcal{R}, \mu)$ is resonant) and that this definition of $\beta$ extend that employed in Definition~\ref{DefRepreFixed}. Also, it follows from the property \ref{P2} of $\lVert \cdot \rVert_{\overline{X}}$ that
	\begin{equation*}
		\lim_{k \to \infty} \lVert f^* \chi_{[k, \infty)} \rVert_{\overline{X}} = \lim_{k \to \infty} \lVert f^* \chi_{[\beta k, \infty)} \rVert_{\overline{X}},
	\end{equation*}
	so it suffices to examine the limit on the right-hand side. 
	
	Next, we note that $t \mapsto (f^* \chi_{[\beta k, \infty)})(t + \beta k)$ is the non-increasing rearrangement of some function in $\mathcal{M}(\mathcal{R}, \mu)$. Indeed, since $f$ itself belongs to $\mathcal{M}(\mathcal{R}, \mu)$, we may construct this function by applying either the operator $T_{\sigma}$ or $T$ from Definition~\ref{DefRepreFixed}, depending on whether the underlying measure space is atomic or not.
	
	Hence, we may apply \eqref{PropIncluCap:E1} onto $f^* \chi_{[\beta k, \infty)}$ to obtain
	\begin{equation*}
		\lVert f^* \chi_{[\beta k, \infty)} \rVert_{\overline{X}} \lesssim \max \left \{ f^*(\beta k), \; \lVert f^* \chi_{[\beta k, \infty)} \rVert_{L^1} \right \}.
	\end{equation*}
	As we assume $f \in L^1$ and it is well known that $L^1_a = L^1$, it is clear that the right-hand side converges to zero as $k \to \infty$.
\end{proof}

We are now suitably equipped to prove our result.

\begin{proof}[Proof of Theorem~\ref{TheoremACHLP}]
	
	As follows from Proposition~\ref{PropAmalAC}, we need to show that $f$ (fixed in the statement of the Theorem) satisfies \eqref{PropAmalAC:Eloc} and \eqref{PropAmalAC:Eglob}.
	
	\eqref{PropAmalAC:Eloc} is simpler, as $f \prec g$ easily implies $f^* \chi_{[0,k^{-1})} \prec g^* \chi_{[0,k^{-1})}$ for all $k$, which is sufficient thanks to Proposition~\ref{PropRepreHLP}.
	
	As for \eqref{PropAmalAC:Eglob}, we assume the contrary, i.e.~that we have a pair of functions $f, g \in X$ such that $g \in X_a$ while
	\begin{align}
		\lim_{k \to \infty} \lVert f^* \chi_{[k, \infty)} \rVert_{\overline{X}} &= \epsilon_0 > 0, \label{TheoremACHLP:E1}
	\end{align}
	and show that this implies $f \not \prec g$.
	
	We assume that $g$ is not the zero function in $X$, as otherwise the conclusion is immediate. In this case we have $X_a \neq \{0\}$, hence it follows from \cite[Theorem~4.15]{MusilovaNekvinda24} that there exists a function $f_0 \in X$ with $f^*(0) = \infty$. Additionally, as the Hardy--Littlewood--P\'{o}lya principle is assumed to hold for $\lVert \cdot \rVert_X$, Proposition~\ref{TheoremNeccHLP} implies $L^1 \cap L^{\infty} \hookrightarrow X$. Therefore, Proposition~\ref{PropIncluCap} together with Proposition~\ref{PropAmalAC}, \eqref{TheoremACHLP:E1}, and \eqref{DefIntL1Infty} imply that $f^* \chi_{[1, \infty)} \notin L^1$ (we have $f^*(1) < \infty$, as follows from $f \in X \hookrightarrow \mathcal{M}_0$, see \cite[Theorem~3.4]{NekvindaPesa24}). Consequently, we may assume that also $g^* \chi_{[1, \infty)} \notin L^1$, because otherwise we would get
	\begin{align*}
		\lim_{k \to \infty} \int_0^{k} f^* \: d\lambda &= \infty, \\
		\lim_{k \to \infty} \int_0^{k} g^* \: d\lambda &< \infty,
	\end{align*}
	which immediately yields $f \not \prec g$. Note that the second estimate employs Theorem~\ref{TheoremNeccHLP}, as this Theorem implies that $\lVert \cdot \rVert_X$ has the property \ref{P5} (see \cite[Theorem~5.7]{Pesa22}, or alternatively \cite[Lemma~2.24]{Pesa22}).
	
	Now, since Proposition~\ref{PropAmalAC} implies for our $g \in X_a$ that
	\begin{equation} \label{TheoremACHLP:E2}
				\lim_{k \to \infty} \lVert g^* \chi_{[k, \infty)} \rVert_{\overline{X}} = 0,
	\end{equation}
	we see that there is a $k_0 \in \mathbb{N}$ such that $2\lVert g^* \chi_{[k_0, \infty)} \rVert_{\overline{X}} < \epsilon_0 \leq \lVert f^* \chi_{[k_0, \infty)} \rVert_{\overline{X}}$. Since the Hardy--Littlewood--P\'{o}lya principle holds for $\lVert \cdot \rVert_X$, this means that there exists a $T_0 \in (k_0, \infty)$ such that
	\begin{equation} \label{TheoremACHLP:E3}
		\int_{k_0}^{T_0} 2g^* \: d\lambda < \int_{k_0}^{T_0} f^* \: d\lambda.
	\end{equation}
	Now, setting
	\begin{equation*}
		\widetilde{T} = \sup \left\{ T \in (k_0, \infty); \; \int_{k_0}^{T} 2g^* \: d\lambda < \int_{k_0}^{T} f^* \: d\lambda \right\},
	\end{equation*}
	we observe from \eqref{TheoremACHLP:E3} that $\widetilde{T} \geq T_0$. Next step is to show that it is, in fact, infinite.
	
	Indeed, if $\widetilde{T}$ were finite, then we would necessarily have
	\begin{equation} \label{TheoremACHLP:E4}
		\int_{k_0}^{\widetilde{T}} 2g^* \: d\lambda = \int_{k_0}^{\widetilde{T}} f^* \: d\lambda.
	\end{equation}
	However, as $\lVert 2g^* \chi_{[\widetilde{T}, \infty)} \rVert_{\overline{X}} \leq 2\lVert g^* \chi_{[k_0, \infty)} \rVert_{\overline{X}} < \epsilon_0 \leq \lVert f^* \chi_{[\widetilde{T}, \infty)} \rVert_{\overline{X}}$, we can find some $T  \in (\widetilde{T}, \infty)$ such that
	\begin{equation} \label{TheoremACHLP:E5}
		\int_{\widetilde{T}}^{T} 2g^* \: d\lambda < \int_{\widetilde{T}}^{T} f^* \: d\lambda.
	\end{equation}
	Putting \eqref{TheoremACHLP:E4} and \eqref{TheoremACHLP:E5} together we obtain a contradiction with the definition of $\widetilde{T}$.
	
	Having established that $\widetilde{T} = \infty$, we combine this fact with our assumption that $g^* \chi_{[1, \infty)} \notin L^1$ to observe that there exists an $T \in (k_0, \infty)$ for which
	\begin{equation} \label{TheoremACHLP:E7}
		\int_{k_0}^T f^* \: d\lambda > \int_{k_0}^T 2g^* \: d\lambda \geq \left(\int_0^{k_0} g^* \: d\lambda - \int_0^{k_0} f^* \: d\lambda \right) + \int_{k_0}^T g^* \: d\lambda.
	\end{equation}
	Here, we make no claim on the sign of the difference in the parentheses on the right-hand side; the only important thing is that it is finite, which follows from the fact that $\lVert \cdot \rVert_X$ has the property \ref{P5} (as discussed above). In fact, if the difference were negative, then we would immediately have $f \not \prec g$; our construction deals with the more interesting case when it is positive.
	
	Finally, it follows from \eqref{TheoremACHLP:E7} that $T$ satisfies
	\begin{equation*}
		\int_0^T f^* \: d\lambda - \int_0^T g^* \: d\lambda \geq \left( \int_{k_0}^T f^* \: d\lambda - \int_{k_0}^T g^* \: d\lambda \right) - \left( \int_0^{k_0} g^* \: d\lambda - \int_0^{k_0} f^* \: d\lambda \right) > 0,
	\end{equation*}
	which shows that $f \not \prec g$.
\end{proof}

We conclude by a corollary of Proposition~\ref{PropAmalAC} which characterises the absolute continuity of the quasinorm in the weak Marcinkiewicz spaces. For an in depth treatment of said spaces we refer the reader to \cite[Section~4]{MusilovaNekvinda24}, here we present only a simple formulation of their definition (which, however, does not lead to any loss of generality). For simplicity and in accordance with the approach taken in \cite{MusilovaNekvinda24} we also restrict ourself to the case of non-atomic infinite measure.

\begin{definition}
	Let $(\mathcal{R}, \mu)$ be non-atomic such that $\mu(\mathcal{R}) = \infty$, let $\lVert \cdot \rVert_X$ be an r.i.~quasi-Banach function norm and $X$ the corresponding r.i.~quasi-Banach function space, and denote by $\varphi_X : [0, \infty) \to [0, \infty)$ the fundamental function of $X$, that is
	\begin{align*}
		\varphi_X(t) &= \lVert \chi_{E_t} \rVert_X, & \text{ for } t \in [0, \infty),
	\end{align*}
	where $E_t \subseteq \mathcal{R}$ is an arbitrary set satisfying $\mu(E_t) = t$.
	
	Then the weak Marcinkiewicz space $m_{\varphi_X}$ is the r.i.~quasi-Banach function space defined via the functional
	\begin{align*}
		\lVert f \rVert_{m_{\varphi_X}} = \sup_{t \in [0, \infty)} \varphi_X(t) f^*(t).
	\end{align*}
\end{definition}

Our construction of the weak Marcinkiewicz space via a fundamental function is chosen for simplicity and briefness as it saves us the need to discuss what functions lead to meaningful functionals (which has been done in \cite[Section~4]{MusilovaNekvinda24}). There is no loss of generality, since the property of being a fundamental function of some r.i.~quasi-Banach function space characterises those functions.

\begin{corollary} \label{CorollaryACwM}
	For the weak Marcinkiewicz space $m_{\varphi_X}$ as defined above the subspace $(m_{\varphi_X})_a$ can be characterised as
	\begin{equation*}
		(m_{\varphi_X})_a = \left \{ f \in m_{\varphi_X}; \; \lim_{t \to 0} \varphi_X(t) f^*(t) = \lim_{t \to \infty} \varphi_X(t) f^*(t) = 0 \right \}.
	\end{equation*}
	
	Specifically, the weak Marcinkiewicz space $m_{\varphi_X}$ does not have absolutely continuous quasinorm (for any valid choice of $\varphi_X$).
\end{corollary}

\begin{proof}
	We first note that since $(\mathcal{R}, \mu)$ is non-atomic and of infinite measure, we have $\overline{m_{\varphi_X}} = m_{\varphi_X}$ where the space on the right hand side is considered over $([0, \infty), \lambda)$.
	
	Sufficiency now follows directly from Proposition~\ref{PropAmalAC} (note that $\varphi_X$ is non-decreasing as follows from the property $\ref{P2}$ of $\lVert \cdot \rVert_X$, see also \cite[Corollary~3.5]{MusilovaNekvinda24}). So does necessity for the limit near zero. 
	
	It remains to prove that $f \in (m_{\varphi_X})_a$ implies
	\begin{equation} \label{CorollaryACwM:E1}
		\lim_{t \to \infty} \varphi_X(t) f^*(t) = 0, 
	\end{equation}
	so let us fix such $f$ and put
	\begin{align*}
		f_t &= \min \{ \lvert f \rvert, \, f^*(t)\} &\text{ on } \mathcal{R}.
	\end{align*}
	Using Proposition~\ref{PropACNimpliesACR} and Proposition~\ref{PropDomConv}, it is easy to check that
	\begin{equation} \label{CorollaryACwM:E2}
		\lim_{t \to \infty} \lVert f_t \rVert_{m_{\varphi_X}} = 0.
	\end{equation}
	Furthermore, it is not hard to verify that
	\begin{equation*}
		f_t^* = \min \{ f^*, \, f^*(t) \} = f^*(t) \chi_{[0,t)} + f^* \chi_{[t, \infty)},
	\end{equation*}
	whence we obtain
	\begin{equation*}
		\varphi_X(t) f^*(t) \leq \lVert f_t \rVert_{m_{\varphi_X}}.
	\end{equation*}
	By combining this estimate with \eqref{CorollaryACwM:E2}, we obtain \eqref{CorollaryACwM:E1}.
	
	Finally, our assumption that $(\mathcal{R}, \mu)$ is non-atomic guarantees  that there is some function $f \in \mathcal{M}(\mathcal{R}, \mu)$ such that $f^* = \frac{1}{\varphi_X}$. Since it is evident that $\frac{1}{\varphi_X} \in \overline{m_{\varphi_X}} \setminus (\overline{m_{\varphi_X}})_a$, Theorem~\ref{TheoremRepreACqN} implies $f \in m_{\varphi_X} \setminus (m_{\varphi_X})_a$.
\end{proof}

\bibliographystyle{dabbrv}
\bibliography{bibliography}
\end{document}